\documentclass[reqno,a4paper]{amsart}
\usepackage{amssymb}
\usepackage{amsmath}
\newcommand{\half}{\frac{1}{2}}
\newcommand{\R}{\mathbb{R}}

\begin{document} 
\newtheorem{prop}{Proposition}[section]
\newtheorem{Def}{Definition}[section] \newtheorem{theorem}{Theorem}[section]
\newtheorem{lemma}{Lemma}[section] \newtheorem{Cor}{Corollary}[section]

\title[Maxwell-Klein-Gordon in Lorenz gauge]{\bf Almost optimal local well-posedness for the Maxwell-Klein-Gordon system with data in Fourier-Lebesgue spaces}
\author[Hartmut Pecher]{
{\bf Hartmut Pecher}\\
Fakult\"at f\"ur  Mathematik und Naturwissenschaften\\
Bergische Universit\"at Wuppertal\\
Gau{\ss}str.  20\\
42119 Wuppertal\\
Germany\\
e-mail {\tt pecher@math.uni-wuppertal.de}}
\date{}

\begin{abstract}
We prove a low regularity local well-posedness result for the Max-\\well-Klein-Gordon system in three space dimensions for data in Fourier - Lebesgue spaces $\widehat{H}^{s,r}$ , where $\|f\|_{\widehat{H}^{s,r}} = \|\langle \xi \rangle^s \widehat{f}(\xi)\|_{\widehat{L}^{r'}}$ , $\frac{1}{r}+\frac{1}{r'} = 1$ . The assumed regularity for the data is almost optimal with respect to scaling as $r \to 1$ . This closes the gap between what is known in the case  $r=2$ , namely $s > \frac{3}{4}$ , and the critical value $s_c = \frac{1}{2}$ with respect to scaling.
\end{abstract}

\maketitle
\renewcommand{\thefootnote}{\fnsymbol{footnote}}
\footnotetext{\hspace{-1.5em}{\it 2000 Mathematics Subject Classification:} 
35Q61, 35L70 \\
{\it Key words and phrases:} Maxwell-Klein-Gordon,  
local well-posedness, Fourier restriction norm method}
\normalsize 
\setcounter{section}{0}
\section{Introduction and main results}
\noindent The Maxwell-Klein-Gordon system couples Maxwell's equation for the electromagnetic field $F_{\mu \nu}: \R^{n+1} \to \R$ with a Klein-Gordon equation for a scalar field $\phi: \R^{n+1} \to \mathbb{C}$ and reads
\begin{align}
\label{1}
\partial^{\nu} F_{\mu \nu} & =  j_{\mu} \\
\label{2}
D^{(A)}_{\mu} D^{(A)\mu} \phi & = m^2 \phi \, ,
\end{align}
where $m>0$ is a constant and
\begin{align}
\label{3}
F_{\mu \nu} & := \partial_{\mu} A_{\nu} - \partial_{\nu} A_{\mu} \\
\label{4}
D^{(A)}_{\mu} \phi & := \partial_{\mu} - iA_{\mu} \phi \\
\label{5}
j_{\mu} & := Im(\phi \overline{D^{(A)}_{\mu} \phi}) = Im(\phi \overline{\partial_{\mu} \phi}) + |\phi|^2 A_{\mu} \, .
\end{align}
Here  $A_{\nu} : {\mathbb R}^{n+1} \to {\mathbb R}$ is the potential. We use the notation $\partial_{\mu} = \frac{\partial}{\partial x_{\mu}}$, where we write $(x^0,x^1,...,x^n) = (t,x^1,...,x^n)$ and also $\partial_0 = \partial_t$ and $\nabla = (\partial_1,...,\partial_n)$. Roman indices run over $1,...,n$ and greek indices over $0,...,n$ and repeated upper/lower indices are summed. Indices are raised and lowered using the Minkowski metric $diag(-1,1,...,1)$.

The Maxwell-Klein-Gordon system describes the motion of a spin 0 particle with mass $m$ self-interacting with an electromagnetic field.

We consider the Cauchy problem in three space dimensions with data
$\phi(x,0) = \phi_0(x)$ , $\partial_t \phi(x,0) $ $= \phi_1(x)$, $F_{\mu \nu}(x,0) =  F^0_{\mu \nu}(x)$ . The potential $A$ is not uniquely determined but one has gauge freedom. The Maxwell-Klein-Gordon equation is namely invariant under the gauge transformation
$\phi \to \phi' = e^{i\chi}\phi$ , $A_{\mu} \to A'_{\mu} = A_{\mu} + \partial_{\mu} \chi$
for any $\chi: {\mathbb R}^{n+1} \to {\mathbb R}$. We exclusively consider the Lorenz gauge $\partial^{\mu} A_{\mu} = 0$. However we remark that even this gauge does not uniquely determine $A$ , because any $\chi$ satisfying $\Box \chi =0$ preserves the Lorenz gauge condition, so that one has to add a further condition (cf. (\ref{10}) below) in order to obtain a unique potential.

We can reformulate the system (\ref{1}),(\ref{2}) under the Lorenz condition
\begin{equation}
\label{6}
\partial^{\mu} A_{\mu} = 0
\end{equation}
 as follows:
$$\square A_{\mu} = \partial^{\nu} \partial_{\nu} A_{\mu} = \partial^{\nu}(\partial_{\mu} A_{\nu} - F_{\mu \nu}) = -\partial^{\nu} F_{\mu \nu} = - j_{\mu} \, , $$
thus (using the notation $\partial = (\partial_0,\partial_1,...,\partial_n)$):
\begin{equation}
\label{16}
\square A = -Im (\phi \overline{\partial \phi}) - A|\phi|^2 =: N(A,\phi) 
\end{equation}
and
\begin{align*}
m^2 \phi & = D^{(A)}_{\mu} D^{(A)\mu} \phi = \partial_{\mu} \partial^{\mu} \phi -iA_{\mu} \partial^{\mu} \phi -i\partial_{\mu}(A^{\mu} \phi) - A_{\mu}A^{\mu} \phi \\
& = \square \phi - 2i A^{\mu} \partial_{\mu} \phi - A_{\mu} A^{\mu} \phi \,
\end{align*}
thus
\begin{equation}
\label{17}
(\square -m^2) \phi = 2i A^{\mu} \partial_{\mu} \phi + A_{\mu} A^{\mu} \phi =: M(A,\phi) \, .
\end{equation}
Conversely, if $\square A_{\mu} = -j_{\mu}$ and $F_{\mu \nu} := \partial_{\mu} A_{\nu} - \partial_{\nu} A_{\mu}$ and the Lorenz condition (\ref{6}) holds then
$$ \partial^{\nu} F_{\mu \nu} = \partial^{\nu}(\partial_{\mu} A_{\nu} - \partial_{\nu} A_{\mu}) = \partial_{\mu} \partial^{\nu} A_{\nu} - \partial^{\nu} \partial_{\nu} A_{\mu} = -\square A_{\mu} = j_{\mu} \, $$
thus (\ref{1}),(\ref{2}) is equivalent to (\ref{16}),(\ref{17}), if (\ref{3}),(\ref{4}) and (\ref{6}) are satisfied.

We assume that the Cauchy data belong to Fourier-Lebesgue spaces:
\begin{equation}
\label{7}
\phi(0)=\phi_0 \in \widehat{H}^{s,r} \quad ,  \quad (\partial_t \phi)(0) = \phi_1 \in \widehat{H}^{s-1,r} \, ,
\end{equation}
\begin{equation}
\label{8}
F_{\mu \nu} = F^0_{\mu \nu} \in \widehat{H}^{s-1,r} \, ,
\end{equation}
where $\|f\|_{\widehat{H}^{s,r}} = \|\langle \xi \rangle^s \widehat{f}(\xi)\|_{{L}^{r'}}$ , $\frac{1}{r}+\frac{1}{r'} = 1$ .

The system (\ref{16}),(\ref{17}) is invariant under the scaling
$$ A_{\lambda}(t,x) = \lambda A(\lambda t,\lambda x) \, , \, \phi_{\lambda}(t,x) = \lambda \phi(\lambda t, \lambda x) \, ,$$
i.e. $A_{\lambda},\phi_{\lambda}$ are solutions, if $A,\phi$ are. The homogeneous Fourier-Lebesgue norm of the initial data scales like
$$\| A_{\lambda}(0,\cdot)\|_{\dot{\widehat{H}}^{s,r}(\R^3)} = \lambda^{s-\frac{3}{r}+1}\| A(0,\cdot)\|_{\dot{\widehat{H}}^{s,r}(\R^3)} \, , $$
so that the critical exponent with respect to scaling is $s_c = \frac{3}{r} -1$ for $A_{\mu}$ as well as $\phi$, thus the critical exponent for $F_{\mu \nu} = \partial_{\mu} A_{\nu} - \partial_{\nu} A_{\mu}$ is $s_c' = \frac{3}{r}- 2$ .  No well-posedness is expected for $ s < s_c$ . Thus, up to the endpoint the best one can hope for is local well-posedness for $ s > s_c$ .

In the classical case $r=2$ and Coulomb gauge $\partial^j A_j = 0$ Klainerman and Machedon \cite{KM} showed global well-posedness in energy space and above, i.e. for data $\phi_0 \in H^s$ , $\phi_1 \in H^{s-1}$ , $a_{0 \nu} \in H^s$ , $\dot{a}_{0 \nu} \in H^{s-1}$ with $s \ge 1$ in $n=3$ dimensions. They made the fundamental observation that the nonlinearities fulfill a null condition. This global well-posedness result was improved by Keel, Roy and Tao \cite{KRT}, who had only to assume $s > \frac{\sqrt 3}{2}$. Local well-posedness for low regularity data was shown by Cuccagna \cite{C} for $s > 3/4$ and small data, but Selberg \cite{S} remarked that this smallness assumption could be removed, all these results for three space dimensions and in Coulomb gauge. Machedon and Sterbenz \cite{MS} proved local well-posedness even in the almost critical range $ s > \half$ , but had to assume a smallness assumption  on the data.

In Lorenz gauge $\partial^{\mu} A_{\mu} = 0$ and data in Sobolev spaces $H^s$,   which was considered much less in the literature, because the nonlinear term $Im(\phi \overline{\partial_{\mu} \phi})$ has no null structure, and three space dimensions the most important progress was made by Selberg and Tesfahun \cite{ST} who were able to circumvent the problem of the missing null condition in the equations for $A_{\mu}$ by showing that the decisive nonlinearities in the equations for $\phi$ as well as $F_{\mu \nu}$ fulfill such a null condition which allows to show that global well-posedness holds for large, finite energy data $s=1$ . The potential possibly loses some regularity compared to the data but as remarked also by the authors this is not the main point because one is primarily interested in the regularity of $\phi$ and $F_{\mu \nu}$ , which both preserve the regularity. Below energy level the author \cite{P2} was also able to show local well-posedness  for large data in this sense, provided $s > \frac{3}{4}$. 

This means that for large data there is still a gap of $\frac{1}{4}$ to $ s_c > \half$ predicted by scaling both in Coulomb and Lorenz gauge. 

We work in Lorenz gauge and close this gap in the limit $r \to 1$ for large data in Fourier-Lebesgue spaces $\widehat{H}^{s,r}$ for $ 1 < r \le 2$ , thus leaving the $H^s$-scale of data ($r=2$). This is remarkable in view of the fact that one of the nonlinearities does not fulfill a null condition. In some sense it is the first large data almost optimal local well-posedness result for the Maxwell-Klein-Gordon equations no matter which gauge is considered.

A null structure in Lorenz gauge was first detected for the Maxwell-Dirac system by d'Ancona, Foschi and Selberg \cite{AFS1}.

In two space dimensions in Coulomb gauge Czubak and Pikula \cite{CP} proved local well-posedness provided that $\phi_0 \in H^s$ , $\phi_1 \in H^{s-1}$ , $a_{0 \nu} \in H^r$ , $\dot{a}_{0 \nu} \in H^{r-1}$,  where $1 \ge s=r > \frac{1}{2}$ or $s=\frac{5}{8}+\epsilon$ , $r=\frac{1}{4}+\epsilon$. 

In four space dimensions Selberg \cite{S} showed local well-posedness in Coulomb gauge for $s>1$ , which is almost critical. In Lorenz gauge
the author \cite{P1} considered also the case $n \ge 4$ and proved local well-posedness for $ s > \frac{n}{2}-\frac{5}{6} $ . 

In order to achieve an almost optimal local well-posedness result for $n=3$ we consider the Lorenz gauge and Cauchy data in Fourier-Lebesgue spaces $\widehat{H}^{s,r}$ for $1 < r \le 2$ . Data in spaces of this type were previously considered by Gr\"unrock \cite{G} and Gr\"unrock-Vega \cite{GV} for KdV and modified KdV equations. Gr\"unrock \cite{G1} used these spaces in order to prove almost optimal low regularity local well-posedness also for wave equations with quadratic derivative nonlinearities for $n=3$. For wave equations with a nonlinearity which fulfills a  null condition this was also shown in the case $n=2$
by Grigoryan-Nahmod \cite{GN}. These results relied on a modification of bilinear estimates which were given by Foschi-Klainerman \cite{FK} in the classical $L^2$-case.

Our solution spaces are generalized Bourgain-Klainerman-Machedon spaces $X^r_{s,b,\pm}$ which were already introduced by Gr\"unrock \cite{G}. They are defined  by its norms
$$ \|\phi\|_{X^r_{s,b,\pm}} = \|\langle \xi \rangle^s \langle \tau \pm |\xi| \rangle^b \widehat{\phi}(\xi,\tau)\|_{L^{r'}_{\tau \xi}}$$
 for $1 < r < \infty$ and $\frac{1}{r}+\frac{1}{r'}=1$ , and 
$$\|\phi\|_{X^r_{s,b,\pm}[0,T]} = \inf_{\bar{\phi}_{|[0,T]} = \phi} \| \bar{\phi} \|_{X^r_{s,b,\pm}} \, . $$
We show that for (admissible) data $\phi(0) \in \widehat{H}^{s,r}$ , $(\partial_t \phi)(0)  \in \widehat{H}^{s-1,r}$ and $(\nabla A_{\mu})(0) \in \widehat{H}^{s-1,r}$ , $(\partial_t A_{\mu})(0) \in \widehat{H}^{s-1,r}$ we obtain a solution of (\ref{16}),(\ref{17}) , where $\phi$ belongs to  $X^r_{s,b,+}[0,T] + X^r_{s,b,-}[0,T] \subset C^0([0,T],\widehat{H}^{s,r}) \cap C^1([0,T],\widehat{H}^{s-1,r})$ for some $b > \frac{1}{r}$ and $s > \frac{5}{2r}-\half$ , so that $s \to 2$ as $r \to 1$ , which is almost optimal with respect to scaling. We also obtain $\nabla A_{\mu} \in X^r_{l,b,+}[0,T] + X^r_{l,b,+}[0,T]$ , where $l > \frac{2}{r}-1$, so that $l \to2$ as $r \to 1$ , but $ l < s $ due to the missing null condition in the term $Im(\phi \overline{\partial \phi})$. However this is of minor interest, because the really important fact is that  $F_{\mu \nu} \in X^r_{s-1,b,+}[0,T] + X^r_{s-1,b,-}[0,T] $ for data $F_{\mu\nu}(0) \in \widehat{H}^{s-1,r}$ , $(\partial_t F_{\mu\nu})(0) \in \widehat{H}^{s-2,r}$. This is a consequence of the fact that $F_{\mu\nu}$ fulfills a wave equation with null forms in the quadratic inhomogeneous terms (see (\ref{27}),(\ref{28}) below).

Fundamental are of course the bi- and trilinear estimates for the nonlinearities where the quadratic terms have null structure except one term, namely $Im(\phi \overline{\partial_{\mu} \phi})$. This is the reason why $A_{\mu}$ (possibly) loses regularity in time. We rely on the bilinear estimates of Foschi-Klainerman \cite{FK}, which were already successfully applied by Gr\"unrock \cite{G1} and Grigoryan-Nahmod \cite{GN}. The general local well-posedness theorem for nonlinear systems of wave equations (and also other types of evolution equations) in $X^r_{s,b}$-spaces , which reduces the problem to multilinear estimates for the nonlinear terms, goes back to Gr\"unrock \cite{G}, cf. also \cite{GN}. For the Cauchy problem for the Maxwell-Klein-Gordon system in Lorenz gauge with $L^2$-based data we rely on the author's paper \cite{P}, which is a refinement of the earlier paper \cite{P2}.

We use the following notation.
Let $\widehat{f}$ denote the Fourier transform of $f$ with respect to space and time as well as with respect to space, which should be clear from the context.
We define the wave-Sobolev spaces $X^r_{s,b,\pm}$ for $1 < r \le 2$ and $\frac{1}{r}+\frac{1}{r'}=1$ as the completion of the Schwarz space $\mathcal{S}({\mathbb R}^{n+1})$ with respect to the norm
$$ \|u\|_{X^r_{s,b,\pm}} = \| \langle \xi \rangle^s \langle  \tau \pm |\xi| \rangle^b \widehat{u}(\tau,\xi) \|_{L^{r'}_{\tau \xi}}  $$
and $X^r_{s,b,\pm}[0,T]$ as the space of the restrictions to $[0,T] \times \mathbb{R}^n$.

We also define the spaces $X^r_{s,b}$ as the completion of  $\mathcal{S}({\mathbb R}^{n+1})$ with respect to the norm
$$ \|u\|_{X^r_{s,b}} =  \| \langle \xi \rangle^s \langle  |\tau| - |\xi| \rangle^b \widehat{u}(\tau,\xi) \|_{L^{r'}_{\tau \xi}}  $$ 
and 
$$ \|u\|_{\dot{X}^r_{s,b}} =  \| |\xi|^s \langle  |\tau| - |\xi| \rangle^b \widehat{u}(\tau,\xi) \|_{L^{r'}_{\tau \xi}}  $$ 
We remark that $\|\phi\|_{X^r_{s,b}} \le \|\phi\|_{X^r_{s,b,\pm}}$ for $b \ge 0$ and the opposite inequality for $b \le 0$ .

Let $\Lambda^{\alpha}$ , $\Lambda_m$ and $D^{\alpha}$  be the multipliers with symbols $\langle \xi \rangle^{\alpha}$ , $(m^2+|\xi|^2)^{\half}$ and 
$|\xi|^{\alpha}$ , respectively, where $ \langle \, \cdot \, \rangle = (1+|\, \cdot \,|^2)^{\frac{1}{2}}$ , $|\nabla| = D$ .

$\Box = -\partial_t^2 + \Delta$ is the d'Alembert operator. $a{\pm} = a\pm\epsilon$ for a sufficiently small $\epsilon > 0$ and $a++=(a+)+$ . \\[0.5em]

Next we formulate our main results. We assume the Lorenz condition
\begin{equation}
\label{6'}
\partial^{\mu} A_{\mu} = 0
\end{equation}
and Cauchy data
\begin{align}
\label{7'}
\phi(x,0) &= \phi_0(x) \in \widehat{H}^{s,r} \quad , \quad \partial_t \phi(x,0) = \phi_1(x) \in \widehat{H}^{s-1,r} \, , \\
\label{8'}
F_{\mu \nu}(x,0)& = F^0_{\mu \nu}(x) \in \widehat{H}^{s-1,r} \, .
\end{align}
We define
\begin{equation}
\label{9}
A_{\nu}(x,0) =: a_{0 \nu}(x) \quad , \quad \partial_t A_{\nu}(x,0) =: \dot{a}_{0 \nu}(x) \, ,
\end{equation} 
which are assumed to fulfill 
\begin{equation}
\label{10}
a_{00} = \dot{a}_{00} = 0 \, , 
\end{equation}
and the following compatibility conditions
\begin{equation}
\label{12}
\partial^k a_{0k} = 0 \, ,
\end{equation}
\begin{equation}
\label{13}
\partial_j a_{0k} - \partial_k a_{0j} = F^0_{jk} \, ,
\end{equation}
\begin{equation}
\label{14}
\dot{a}_{0k} = F^0_{0k} \, , 
\end{equation}
\begin{equation}
\label{15}
\partial^k F^0_{0k}  = Im(\phi_0 \overline{\phi}_1) \, .
\end{equation}
(\ref{10}) may be assumed because otherwise the Lorenz condition does not determine the potential uniquely. As remarked already any function  $\chi$ in the gauge transformation with $\Box \chi =0$ preserves the Lorenz condition. Thus in order to obtain uniqueness we assume that $\chi$ moreover fulfills $\Delta \chi(0) = - \partial^j a_{0j}$ and $(\partial_t \chi)(0) = - a_{00}$ . This implies by the gauge transformation $a'_{00} = a_{00} + (\partial_t \chi)(0) = 0$ , $\partial^j a'_{0j} = \partial^j a_{0j} + \partial^j \partial_j \chi(0) = 0 $,  so that the Lorenz condition implies $\dot{a}_{00}   = 0$ .

(\ref{12}) follows from the Lorenz condition (\ref{6}) in connection with (\ref{10}). (\ref{13}) follows from (\ref{3}), similarly (\ref{14}) from (\ref{3}) and (\ref{10}), thus $\dot{a}_{0k}$ is uniquely determined. (\ref{1}) requires
$$ \partial^k F^0_{0k} = j_0(0) = Im(\phi_0 \overline{\phi}_1) + |\phi_0|^2 a_{00} = Im(\phi_0 \overline{\phi}_1) \, $$
thus (\ref{15}). By (\ref{12}) we have
$$ \Delta a_{0j} = \partial^k \partial_k a_{0j} = \partial^k(\partial^j a_{0k} - F^0_{jk}) = - \partial^k F^0_{jk} \, , $$
so that $a_{0j}$ is uniquely determined as
$$ a_{0j} = (-\Delta)^{-1} \partial^k F^0_{jk} \, .$$ 
These conditions imply the following regularity for the initial data
\begin{equation}
\label{11}
\nabla a_{0j} \in \widehat{H}^{s-1,r} \, , \,  \dot{a}_{0j} \in \widehat{H}^{s-1,r} \, .
\end{equation}

We prefer to rewrite our system (\ref{16}),(\ref{17}) as a first order (in t) system.
Let $\phi_{\pm}=\frac{1}{2}(\phi \pm (i\Lambda_m)^{-1}\phi_t)$ , so that $\phi=\phi_+ + \phi_-$ and $ \partial_t \phi = i \Lambda_m(\phi_+ - \phi_-)$, and $A_{\pm} = \frac{1}{2}(A \pm (iD)^{-1} A_t)$ so that $A=A_+ + A_-$ and $\partial_t A = iD(A_+-A_-)$ . We obtain the equivalent system
\begin{align}
\label{3.1}
(i\partial_t \pm \Lambda_m) \phi_{\pm} & = - (\pm 2 \Lambda_m)^{-1} \mathcal{M}(\phi_+,\phi_-,A_+,A_-) \\
\label{3.2}
(i\partial_t \pm D)A_{\pm} & = -(\pm 2D)^{-1} \mathcal{N}(\phi_+,\phi_-,A_+,A_-) \, ,
\end{align}
where
\begin{align}
\label{3.3}
\mathcal{M}(\phi_+,\phi_-,A_+,A_-) & = A^{\mu} \partial_{\mu} \phi + A_{\mu} A^{\mu} \phi \\
\label{3.4}
\mathcal{N}_0(\phi_+,\phi_-,A_+,A_-) & = Im(\phi i \Lambda_m(\overline{\phi}_+ - \overline{\phi}_-)) - A_0 |\phi|^2 \\
\label{3.5}
\mathcal{N}_j(\phi_+,\phi_-,A_+,A_-) & = -Im(\phi \overline{\partial_j \phi}) -A_j |\phi|^2 \, .
\end{align}
The initial data are given by
\begin{align}
\label{3.5a}
\phi_{\pm}(0) & = \frac{1}{2}(\phi_0 \pm (i \Lambda_m)^{-1} \phi_1) \in \widehat{H}^{s,r}\\
\label{3.5b}
A_{0\pm}(0)& = \frac{1}{2}(a_{00}\pm(iD^{-1})\dot{a}_{00}) = 0 \\
\label{3.5c}
A_{j\pm}(0) & = \frac{1}{2}(a_{0j} \pm (iD)^{-1} \dot{a}_{0j})  \, , \,
  \nabla A_{j\pm}(0) \in \widehat{H}^{l-1,r} \quad {\mbox for}\,\, l \le s\, .
\end{align}
(\ref{3.5b}) follows from (\ref{10}). The regularity follows by (\ref{7}) and (\ref{11}). 

We split $A_{\pm} = A_{\pm}^{hom} + A_{\pm}^{inh}$ into its homogeneous and inhomogeneous part, where $(i\partial_t \pm D)A_{\pm}^{hom} =0$ with data as in (\ref{3.5b}) and (\ref{3.5c}) and $A_{\pm}^{inh}$ is the solution of (\ref{3.2}) with zero data. \\[0.5em]

Our first main theorem reads as follows:
\begin{theorem}
\label{Theorem0.1}
Let $1 < r \le 2$ , and assume $ s = \frac{5}{2r}-\half+\delta$ , $ l = \frac{3}{r}-1+\delta $ , where $\delta > 0$ . 
The data are assumed to fulfill (\ref{7'})-(\ref{15}). Then there exists $T > 0$,  $T=T(\|\phi_{\pm}(0)\|_{\widehat{H}^{s,r}},\|\nabla A_{j\pm}(0)\|_{\widehat{H}^{l-1,r}})$ , such that the problem (\ref{3.1})-(\ref{3.5c}) has a unique local solution
$$ \phi_{\pm} \in X^r_{s,b,\pm}[0,T]  \, , \,
\nabla A^{hom}_{\pm} \in X^r_{l-1,1-\epsilon_0,\pm}[0,T] \, , \, A^{inh}_{\pm} \in X^r_{l,1-\epsilon_0,\pm}[0,T] \, ,$$
 where $ b > \frac{1}{r}$ and $\epsilon_0$ is a small positive number. 
\end{theorem}
{\bf Remark:} The solution depends continuously on the initial data and persistence of higher regularity holds (see Theorem \ref{Theorem0.0} below). \\[0.5em]

In order to obtain the optimal regularity for $F_{\mu\nu}$ it is possible to derive from Maxwell's equations (\ref{1}) and (\ref{5}) the following wave equations, where we refer to \cite{S}, section 3.2 or \cite{P}, section 2.
\begin{align}
\label{27}
\square F_{k0} & =  Im(\partial_t \phi\overline{\partial_k \phi} - \partial_k \phi \overline{\partial_t \phi}) + \partial_t(A_k |\phi|^2) - \partial_k(A_0 |\phi|^2)
\end{align}
and
\begin{align}
\label{28}
\square F_{kl} & =Im(\partial_l \phi\overline{\partial_k \phi} - \partial_k \phi \overline{\partial_l \phi}) + \partial_l(A_k |\phi|^2) - \partial_k(A_l |\phi|^2) \, .
\end{align}

We prove as a consequence of this system or its equivalent first order (in t) system the second main result.
\begin{theorem}
\label{Theorem0.2}
Let $ 1 < r \le 2 $ and $ s = \frac{5}{2r} - \half+\delta $ , where $ \delta > 0 $, and $ b = \frac{1}{r}+$ . The data are assumed to fulfill (\ref{7'})-(\ref{15}).  Let $\phi$ , $A_{\mu}$
be the solution of Theorem \ref{Theorem0.1}. Then
$$ \nabla F_{\mu\nu} \,,\, \partial_t F_{\mu\nu} \in X^r_{s-2,b}[0,T]  \, . $$
\end{theorem}

The final result as a consequence of Theorem \ref{Theorem0.1} and Theorem \ref{Theorem0.2} reads as follows.
\begin{theorem}
\label{Theorem0.3} 
Let $1 < r \le 2$ , and assume $ s = \frac{5}{2r}-\half+\delta$ , $ l = \frac{3}{r}-1+\delta $ , where $\delta > 0$ . 
The data are assumed to fulfill (\ref{7'})-(\ref{15}). Then there exists $T > 0$  , $ b > \frac{1}{r} $, such that the problem (\ref{1})-(\ref{5}) with Lorenz condition (\ref{6}) and Cauchy data (\ref{7}),(\ref{8})  has a unique local solution
$$ \phi \in X^r_{s,b,+}[0,T]+X^r_{s,b,-}[0,T]  \, , \, \partial_t \phi \in X^r_{s-1,b,+}[0,T]+X^r_{s-1,b,-}[0,T] $$
and
$$\nabla F_{\mu\nu} \,,\, \partial_t F_{\mu\nu} \in X^r_{s-2,b}[0,T] $$
relative to a potential $A=(A_0,A_1,A_2,A_3)$ , where $A=A_+^{hom}+A_-^{hom} + A_+^{inh} + A_-^{inh}$ with
$\nabla A^{hom}_{\pm} \in X^r_{l-1,1-\epsilon_0,\pm}[0,T] \, , \, A^{inh}_{\pm} \in X^r_{l,1-\epsilon_0,\pm}[0,T] \, ,$
 where $\epsilon_0$ is a small positive number. 
\end{theorem}

A consequence of Theorem \ref{Theorem0.3} is
\begin{Cor}
\label{Corollary}
$\phi $ and $ F_{\mu \nu}$ have the regularity
\begin{align*}
\phi& \in C^0([0,T],\widehat{H}^{s,r}(\R^3)) \cap C^1([0,T],\widehat{H}^{s-1,r}(\R^3))      \, , \\
F_{\mu \nu} &\in C^0([0,T],\widehat{H}^{s-1,r}(\R^3)) \cap C^1([0,T],\widehat{H}^{s-2,r}(\R^3)) \, . 
\end{align*}
\end{Cor}
\begin{proof}
By (\ref{102}) below we immediately obtain $\phi \in C^0([0,T],\widehat{H}^{s,r})$ and $\nabla F_{\mu \nu}$ , $\partial_t F_{\mu \nu} \in C^0([0,T],\widehat{H}^{s-2,r})$. Using $F_{\mu \nu}(0) \in \widehat{H}^{s-1,r}$ this implies also $F_{\mu \nu} \in C^0([0,T],$ $\widehat{H}^{s-2,r})$ , which gives the claimed regularity of $F_{\mu \nu}$.
\end{proof}

For the following basic properties of $X^r_{s,b}$-spaces as well as a general local well-posedness theorem for nonlinear systems we refer to Gr\"unrock \cite{G}.

The transfer principle is the following result (cf. \cite{ KS}, Prop. 3.5).
\begin{prop}
\label{Prop.0.1}
If $ T: \widehat{H}^{s_1,r} \times ... \times \widehat{H}^{s_k,r} \to \widehat{H}^{s,r}$ is k-linear operator, $b > \frac{1}{r}$ and
$$ \|T(e^{\pm_1 itD} f_1,...,e^{\pm_k itD} f_k)\|_{\widehat{L}^p_t(\widehat{L}^q_x)} \le c \|f_1\|_{\widehat{H}^{s_1,r}} ... \|f_k\|_{\widehat{H}^{s_k,r}} $$
for all signs $\pm_1,...,\pm_k$ , where $\widehat{L}^r= \widehat{H}^{0,r}$ , then
$$ \|T(u_1,...,u_k)\|_{\widehat{L}^p_t(\widehat{L}^q_x)} \le c \|u_1\|_{X^r_{s_1,b}} ... \|u_k\|_{X^r_{s_k,b}} \, \forall \, u_j \in X^r_{s_j,b} \, . $$
\end{prop}

For general phase functions $\phi: \R^n \to \R$ (in our case $\phi(\xi) = \pm |\xi|$) we define
$$ \|f\|_{X^r_{s,b,\phi}} = \| \langle \xi \rangle^s \langle \tau - \phi(\xi)\rangle \|_{L^{r'}_{\tau \xi}} \, , $$
where $1 < r < \infty$ , $\frac{1}{r}+\frac{1}{r'} =1$ , $ b \in \R $ . Then these spaces have the following properties:
\begin{equation}
\label{102}
(X^r_{s,b,\phi})^* = X^{r'}_{-s,-b,\phi} \quad , \quad X^r_{s,b,\phi} \subset C^0(\R,\widehat{H}^{s,r}) \quad{\mbox for}\,\, b > \frac{1}{r} \,, 
\end{equation}
$(X^{r_0}_{s_0,b_0,\phi},X^{r_1}_{s_1,b_1\phi})_{[\Theta]} = X^r_{s,b,\phi}$ (complex interpolation space) , where $\frac{1}{r}= \frac{1-\Theta}{r_0} + \frac{\Theta}{r_1}$, $s = (1-\Theta)s_0+\Theta s_1$ , $ b= \frac{1-\Theta}{s_0} + \frac{\Theta}{b_1}$ , $0 < \Theta < 1$ .

The following general local well-posedness theorem is an obvious generalization of \cite{G}, Thm. 1.
\begin{theorem}
\label{Theorem0.0}
Assume that for $s\in\R$ , $1<r<\infty$ there exists $ b > \frac{1}{r}$ such that
$$ \|N(u)\|_{X^r_{s,b-1+,\phi}} \le \omega( \|u\|_{X^r_{s,b,\phi}})  \|u\|_{X^r_{s,b,\phi}}\, ,$$
where $N$ is a smooth function, and
$$ \|N(u)-N(v)\|_{X^r_{s,b-1+,\phi}} \le \omega_1 (\|u\|_{X^r_{s,b,\phi}}+\|v\|_{X^r_{s,b,\phi}}) \|u-v\|_{X^r_{s,b,\phi}} \, .$$
Here $\omega$ and $\omega_1$ are increasing smooth functions.
Then the Cauchy problem
$$ \partial_t u - i \phi(D) u = N(u) \quad , \quad u(0) = u_0 \in \widehat{H}^{s,r} $$
has a unique local solution $ u \in X^r_{s,b,\phi}[0,T]$ , where $T=T(\|u_0\|_{\widehat{H}^{s,r}}) > 0$ . This solution depends locally lipschitzian on the data and higher regularity is preserved.
\end{theorem}
\noindent{\bf Remark:} This theorem can be generalized to systems of equations in a straightforward manner, especially to the Maxwell-Klein-Gordon system in the form (\ref{3.1}),(\ref{3.2}).

\section{Estimates for the nonlinearities}
The null forms of Klainerman-Machedon are defined as follows:
\begin{align*}
Q_{ij}(u,v) &= \partial_i u \partial_j v - \partial_j u \partial_i v   &q_{ij}(D u,D v)= Q_{ij}(u,v) \, ,\\
Q_{0j}(u,v) &= \partial_t u \partial_j v - \partial_j u \partial_t v   &q_{0j}(D u, D v) = Q_{0j}(u,v) \, ,\\
Q_0(u,v) &= \partial_t u \partial_t v - \nabla u \cdot \nabla v   &q_0(D u, D v) = Q_0(u,v) \, .
\end{align*}

\begin{lemma}
\label{Prop.1.1}
Let $1 < r \le 2$ , $b > \frac{1}{r} $  , $\alpha_1+\alpha_2 > \frac{3}{r}-1$ , $\alpha_0,\alpha_1,\alpha_2 \ge 0$ , $\alpha_1+\alpha_2-\alpha_0 = \frac{2}{r}$ , $\alpha_1,\alpha_2 \le \frac{2}{r}$ and $\alpha_1,\alpha_2 \neq \frac{2}{r}-\half$ . The following estimate applies
$$ \|q_{ij}(u,v)\|_{\dot{X}^r_{\alpha_0,0}} \lesssim \|u\|_{\dot{X}^r_{\alpha_1,b}} \|v\|_{\dot{X}^r_{\alpha_2,b}} \, . $$
\end{lemma}
\begin{proof}
We decompose $uv=u_+ v_++ u_+v_-+u_-v_++u_-v_-$ , where $u_{\pm}(t)=e^{\pm itD} f$ and $v_{\pm}(t)=e^{\pm itD} g$ . It is sufficient to consider $u_+v_+$ and $u_+v_-$ . Using
$$ \widehat{u_{\mp}}(\tau,\xi) = c \delta(\tau \pm |\xi|) \widehat{f}(\xi) \, , \,  \widehat{v´_{\mp}}(\tau,\xi) = c \delta(\tau \pm |\xi|) \widehat{g}(\xi) \, . $$
The Fourier symbol of $q_{ij}$ is bounded by $|\frac{\eta \times (\xi - \eta)}{|\eta| |\xi - \eta|}|$ , so that in order to estimate  $\|q_{ij}(u_+, v_{\pm})\|_{\dot{X}^r_{\alpha_0,0}}$ we have to consider
\begin{align}
\nonumber
&\left\| \int |\xi|^{\alpha_0}\left|\frac{\eta \times (\xi - \eta)}{|\eta| |\xi - \eta|}\right| \delta(\tau-|\eta| \mp |\xi - \eta|) \widehat{f}(\eta) \widehat{g}(\xi-\eta) d\eta \right\|_{L^{r'}_{\tau \xi}}^{r'} \\
\nonumber
&\lesssim \sup_{\xi,\tau} (\int |\xi|^{r \alpha_0} \left|\frac{\eta \times (\xi - \eta)}{|\eta| |\xi - \eta|}\right|^r \delta(\tau-|\eta| \mp |\xi - \eta|) |\xi - \eta|^{-\alpha_1 r} |\eta|^{-\alpha_2 r} d \eta)^{\frac{r'}{r}} \\
\label{99}
& \hspace{5em} \cdot\|\widehat{D^{\alpha_1} f}\|_{L^{r'}} \|\widehat{D^{\alpha_2} g}\|_{L^{r'}} \, .
\end{align}
In the elliptic case, where we have the factor $\delta(\tau-|\eta|-|\xi - \eta|)$  , we use the following estimate:
$$ \left|\frac{\eta \times (\xi - \eta)}{|\eta| |\xi - \eta|}\right| \lesssim \frac{|\xi|^{\half}(|\eta|+|\xi-\eta|-|\xi|)^{\half}}{|\eta|^{\half} |\xi - \eta|^{\half}} = \frac{|\xi|^{\half}(\tau-|\xi|)^{\half}}{|\eta|^{\half} |\xi - \eta|^{\half}} \, ,$$
and in the hyperbolic case with the factor $\delta(\tau-|\eta|+|\xi - \eta|)$ :
$$ \left|\frac{\eta \times (\xi - \eta)}{|\eta| |\xi - \eta|}\right| \lesssim \frac{|\xi|^{\half}(|\xi|-||\eta|-|\xi-\eta||)^{\half}}{|\eta|^{\half} |\xi - \eta|^{\half}} = \frac{|\xi|^{\half}(||\tau|-|\xi||)^{\half}}{|\eta|^{\half} |\xi - \eta|^{\half}} $$ 
 (cf. \cite{FK}, Lemma 13.2).
Thus we have to prove
$\sup_{\xi,\tau} I_{\pm}(\xi,\tau) \lesssim 1$ , where
$$I_{\pm}:=|\xi|^{\alpha_0 r +\frac{r}{2}} ||\tau| - |\xi||^{\frac{r}{2}} \int \frac{\delta(\tau-|\eta|\mp |\xi-\eta|)}{|\xi - \eta|^{(\alpha_1+\half)r} |\eta|^{(\alpha_2+\half)r}} d\eta \, . $$
{\bf Elliptic case.}
By \cite{FK}, Prop. 4.3 we obtain
$$\int \frac{\delta(\tau-|\eta|-|\xi-\eta|)}{|\xi - \eta|^{(\alpha_1+\half)r} |\eta|^{(\alpha_2+\half)r}} d\eta \sim \tau^A (\tau - |\xi|)^B \, , $$
where we assume without loss of generality $\alpha_1 \ge \alpha_2$ , so that in the case $\alpha_1 > \frac{2}{r}-\half$ we obtain $A= \max((\alpha_1+\half)r,2) -(\alpha_1+\alpha_2+1)r = -(\alpha_2+\half)r$ and $B=2-(\alpha_1+\half)r$. This implies
$$ I_+ \lesssim |\xi|^{(\alpha_0 + \half)r} \tau^{-(\alpha_1+\half)r} |\tau -|\xi||^{2-\alpha_1r} \lesssim \tau^{(\alpha_0-\alpha_1-\alpha_2)r +2} = 1 $$ by our assumptions $\alpha_1 \le \frac{2}{r}$ and  $\alpha_1+\alpha_2-\alpha_0 = \frac{2}{r}$ and $\alpha_0,\alpha_1,\alpha_2 \ge 0$ , where we used $\tau = |\eta|+|\xi - \eta| \ge |\xi|$ . In the case $\alpha_1 < \frac{2}{r}-\half$ we obtain similarly $A=2-(\alpha_1+\alpha_2+1)r$, $B=0$ , so that
$$ I_+ \lesssim |\xi|^{(\alpha_0+\half)r} \tau^{2-(\alpha_1 + \alpha_2+1)r} |\tau-|\xi||^{\frac{r}{2}} \lesssim \tau^{(\alpha_0-\alpha_1-\alpha_2)r+2} = 1 \, . $$
{\bf Hyperbolic case.}\\
Subcase: $|\eta| + |\xi-\eta| \le 2|\xi|$ .\\
By \cite{FK}, Prop. 4.5 we obtain
$$\int_{|\eta|+|\xi-\eta| \le 2 |\xi|} \frac{\delta(\tau-|\eta|+|\xi-\eta|)}{|\xi - \eta|^{(\alpha_1+\half)r} |\eta|^{(\alpha_2+\half)r}} d\eta \sim |\xi|^A (|\tau| - |\xi|)^B \, , $$
where in the case $0 \le \tau \le |\xi|$ : $A=\max((\alpha_2+\half)r,2) -(\alpha_1+\alpha_2+1)r$ , $B= 2-\max((\alpha_2+\half)r,2)$ . 
\\
If $\alpha_2 > \frac{2}{r}-\half$ we obtain $A=-(\alpha_1+\half)r$ , $B= 2-(\alpha_2+\half)r$ , so that
$$
I_-  \lesssim |\xi|^{(\alpha_0 +\half)r} |\xi|^{-(\alpha_1+\half)r} ||\tau|-|\xi||^{2-\alpha_2 r} \lesssim |\xi|^{2+(\alpha_0-\alpha_1-\alpha_2)r} = 1 \, , $$
by our assumptions $\alpha_2 \le \frac{2}{r}$ , $\alpha_1+\alpha_2-\alpha_0 = \frac{2}{r}$ and $\alpha_0,\alpha_1,\alpha_2 \ge 0$ . 
If $\alpha_2 < \frac{2}{r}-\half$ we have $A=2-(\alpha_1+\alpha_2+1)r$ , $B=0$ , so that
$$I_- \lesssim |\xi|^{(\alpha_0 +\half)r} ||\xi|-|\tau||^{\frac{r}{2}} |\xi|^{2-(\alpha_1+\alpha_2+1)r} \lesssim 1 \, . $$
In the case $-|\xi| \le \tau \le 0$ we have $A=\max((\alpha_1+\half)r,2) -(\alpha_1+\alpha_2+1)r$ , $B=2-\max((\alpha_1+\half)r,2)$ . Similarly as before we obtain $I_- \lesssim 1$ .\\
Subcase: $|\eta|+|\xi-\eta| \ge 2|\xi|$ . \\
We use \cite{FK}, Lemma 4.4 and obtain
\begin{align*}
&\int_{|\eta|+|\xi-\eta| \ge 2 |\xi|} \frac{\delta(\tau-|\eta|+|\xi - \eta|)}{|\eta|^{(\alpha_1+\half)r} |\xi-\eta|^{\alpha_2+\half)r}} d\eta \\
& \sim \int_2^{\infty} (|\xi|x+\tau)^{-(\alpha_2+\half)r} (|\xi|x-\tau)^{-(\alpha_1+\half)r} (|\xi|^2x^2 - \tau^2) dx \\
& \sim \int_2^{\infty}(x+\frac{\tau}{|\xi|})^{-(\alpha_2+\half)r+1} (x-\frac{\tau}{|\xi|})^{-(\alpha_1+\half)r+1} dx \cdot |\xi|^{-(\alpha_1+\alpha_2+1)r+2} \, . 
\end{align*}
The lower bound of the integral is in fact 2, as the proof of \cite{FK}, Lemma 4.4 shows.
Using $|\tau| \le |\xi|$ the integral is bounded, provided $\alpha_1+\alpha_2 > \frac{3}{r}-1$ , and  we obtain
$$ I_- \lesssim |\xi|^{\alpha_0r+\frac{r}{2}-(\alpha_1+\alpha_2+1)r+2} ||\xi|-|\tau||^{\frac{r}{2}} \lesssim |\xi|^{(\alpha_0-\alpha_1-\alpha_2)r+2} = 1 \, ,$$
by our assumption $\alpha_1+\alpha_2-\alpha_0 = \frac{2}{r}$ .

Finally, combining the estimate for $I_{\pm}$ and (\ref{99}) with the transfer principle Prop. \ref{Prop.0.1} we obtain the result.
\end{proof}

\begin{lemma}
\label{Prop.1.2}
Let the assumptions of Lemma \ref{Prop.1.1} be satisfied. The following estimate applies
$$\| q_{0j}(u,v)\|_{\dot{X}^r_{\alpha_0,0}} \lesssim \|u\|_{\dot{X}^r_{\alpha_1,b}} \|v\|_{\dot{X}^r_{\alpha_2,b}} \, . $$
\end{lemma}
\begin{proof} The symbol of $q_{0j}$ can be estimated by \cite{FK}, Lemma 13.2 as follows:
$$ |\widehat{q_{0j}} (\eta,\xi-\eta)| \lesssim \frac{(|\eta|+|\xi-\eta|)^{\half} (|\eta|+|\xi-\eta|-|\xi|)^{\half}}{|\eta|^{\half} |\xi-\eta|^{\half}} = \frac{\tau^{\half} (\tau-|\xi|)^{\half}}{|\eta|^{\half} |\xi-\eta|^{\half}} $$
in the elliptic case, and
$$ |\widehat{q_{0j}} (\eta,\xi-\eta)| \lesssim \frac{|\xi|^{\half} (|\xi|-||\eta|-|\xi-\eta||)^{\half}}{|\eta|^{\half} |\xi-\eta|^{\half}} = \frac{|\xi|^{\half} ||\xi|-|\tau||^{\half}}{|\eta|^{\half} |\xi-\eta|^{\half}} $$
in the hyperbolic case. In the hyperbolic case we obtain the same bounds as for $q_{ij}$ in Lemma \ref{Prop.1.1}, whereas in the elliptic case we have to replace one factor $|\xi|^{\frac{r}{2}}$ by $\tau^{\frac{r}{2}}$. Thus we have to show $\sup_{\xi,\tau} I' \lesssim 1$ , where
$$I' = |\xi|^{\alpha_0 r} |\tau|^{\frac{r}{2}} |\tau-\xi|^{\frac{r}{2}} \int \frac{\delta(\tau-|\eta|-|\xi-\eta|)}{|\xi-\eta|^{(\alpha_1+\half)r} |\eta|^{(\alpha_2+\half)r}} d\eta \, . $$
As in Lemma \ref{Prop.1.1} we obtain in the case $\alpha_1 > \frac{2}{r}-\half$ :
$$ I' \lesssim |\xi|^{\alpha_0r}\tau^{\frac{r}{2}} \tau^{-(\alpha_2+\half)r} ||\tau|-|\xi||^{2-\alpha_1r} \lesssim \tau^{-(\alpha_1 + \alpha_2 -\alpha_0)r+2} = 1 \, ,$$
if we assume $\alpha_1+\alpha_2-\alpha_0 = \frac{2}{r}$ . In the case $\alpha_1 <\frac{2}{r}-\half$ we obtain similarly
$$ I' \lesssim |\xi|^{\alpha_0r}\tau^{\frac{r}{2}} \tau^{2-(\alpha_1+\alpha_2+1)r} ||\tau|-|\xi||^{\frac{r}{2}} \lesssim \tau^{-(\alpha_1 + \alpha_2 -\alpha_0)r+2} = 1 \, .$$
The transfer principle completes the proof.
\end{proof}

\begin{lemma}
\label{Prop.1.3}
Let $1 < r \le 2$ , $b > \frac{1}{r}$ , $\alpha_0,\alpha_1,\alpha_2 \ge 0$ , $\alpha_1,\alpha_2 \le \frac{2}{r}$ , $\alpha_1+\alpha_2 \ge \frac{3}{r}-2$, $\alpha_1+\alpha_2-\alpha_0 = \frac{2}{r}$ and $\alpha_1,\alpha_2 \neq \frac{2}{r}-1$ . The following estimate applies:
$$\| q_0(u,v)\|_{\dot{X}^r_{\alpha_0,0}} \lesssim \|u\|_{\dot{X}^r_{\alpha_1,b}} \|v\|_{\dot{X}^r_{\alpha_2,b}} \, . $$
\end{lemma}
\begin{proof}
By \cite{FK}, Lemma 13.2 we obtain
$$ |\widehat{q_0}(\eta,\xi-\eta)| \sim \frac{(|\eta| + |\xi-\eta|)(|\eta|+|\xi-\eta|-|\xi|)}{|\eta| |\xi-\eta|} $$ 
and
$$ |\widehat{q_0}(\eta,\xi-\eta)| \sim \frac{|\xi| (|\xi|-||\eta|-|\xi-\eta||)}{|\eta| |\xi-\eta|} \, .$$ 
{\bf Elliptic part:}
We use the first bound for the elliptic part,  we have $\tau=|\eta|+|\xi-\eta|$ so that $\tau \ge |\xi|$ , and we have to show
$$I:= |\xi|^{\alpha_0 r} \tau^r |\tau-|\xi||^r \int \frac{\delta(\tau-|\eta|-|\xi-\eta|)}{|\eta|^{(\alpha_1+1)r} |\xi-\eta|^{(\alpha_2+1)r}} d\eta \lesssim 1\, . $$
By \cite{FK}, Lemma 4.3 the integral behaves like $\tau^A |\tau - |\xi||^B$ , where $A=\max((\alpha_1+1)r,2)-(\alpha_1+\alpha_2+2)r = -(\alpha_2 +1)r$ , $B=2-(\alpha_1+1)r$ , if $\alpha_1 > \frac{2}{r} -1$ , so that
$$I  \sim |\xi|^{\alpha_0 r} \tau^r |\tau -|\xi||^r \tau^{-(\alpha_2+1)r} |\tau - |\xi||^{2-(\alpha_1+1)r} \lesssim \tau^{(\alpha_0-\alpha_1-\alpha_2)r+2} = 1 \,,$$
if we assume $\alpha_1 \le \frac{2}{r}$ , $\alpha_1+\alpha_2-\alpha_0 = \frac{2}{r}$ and $\alpha_0 \ge 0$ .  In the case $\alpha_1 < \frac{2}{r}-1$ we have $A= 2-(\alpha_1+\alpha_2+2)r$ , $B=0$ , so that
$$ I = |\xi|^{\alpha_0r} \tau^r |\tau-|\xi||^r \tau^{2-(\alpha_1+\alpha_2+2)r} \lesssim 1 $$
as before. \\
{\bf Hyperbolic part:}
Subcase $|\eta|+|\xi-\eta| \le 2|\xi|$ . \\
Using the second bound for $\widehat{q_0}$ we have to prove
$$ I:= |\xi|^{\alpha_0 r} |\xi|^r ||\tau|-|\xi||^r \int_{|\eta|+|\xi-\eta| \le 2|\xi|} \frac{\delta(\tau - |\eta|+|\xi-\eta|)}{|\eta|^{(\alpha_1+1)r} |\xi-\eta|^{(\alpha_2+1)r}} d\eta \lesssim 1 \, . $$
By \cite{FK}, Lemma 4.5 the integral behaves as follows: \\
a. if $0 \le \tau \le |\xi|$ like $|\xi|^A (|\xi|-\tau)^B $ , where $A=\max((\alpha_2+1)r,2)-(\alpha_1+\alpha_2+2)r$ and $B= 2-\max((\alpha_2+1)r$ . In the case $\alpha_2 > \frac{2}{r}-1$ we obtain $A=-(\alpha_1+1)r$ and $B= 2-(\alpha_2+1)r$ , so that
\begin{align*}
I &\sim |\xi|^{\alpha_0r} |\xi|^r ||\tau|-|\xi||^r |\xi|^{-(\alpha_1 +1)r} (|\xi|-\tau)^{2-(\alpha_2+1)r} \\
&= |\xi|^{(\alpha_0-\alpha_1)r} ||\tau|-|\xi||^{2-\alpha_2r} \lesssim |\xi|^{(\alpha_0-\alpha_1-\alpha_2)r+2} = 1  
\end{align*}
by the assumptions $\alpha_2 \le \frac{2}{r}$ and $\alpha_1+\alpha_2-\alpha_0 = \frac{2}{r}$ . In the case $\alpha_2 < \frac{2}{r}-1$ we obtain $A=2-(\alpha_1+\alpha_2+2)r$ , $B=0$ , thus
$$I \lesssim |\xi|^{(\alpha_0+1)r} ||\tau|-|\xi||^r |\xi|^{2-(\alpha_1+\alpha_2+2)r} \lesssim |\xi|^{(\alpha_0-\alpha_1-\alpha_2)r+2} = 1 \, .$$
b. If $-|\xi|< \tau \le 0$ the integral behaves like $|\xi|^A (|\xi|+\tau)^B$ , where $A=\max((\alpha_1+1)r,2) -(\alpha_1+\alpha_2+2)r$ , $B=2-\max((\alpha_1+1)r,2) $ ,  so that in the case $\alpha_1 < \frac{2}{r}-1$ we obtain $A=2-(\alpha_1+\alpha_2+2)r$, $B=0$ . Therefore 
$$I \lesssim |\xi|^{(\alpha_0+1)r} ||\tau|-|\xi||^r |\xi|^{2-(\alpha_1+\alpha_2+2)r} \lesssim 1 \, . $$
If $\alpha_1 > \frac{2}{r}-1$ we obtain with $A=-(\alpha_2+1)r$ , $B=2-(\alpha_1+1)r$ the estimate
$$I \lesssim |\xi|^{(\alpha_0+1)r} ||\tau|-|\xi||^r |\xi|^{-(\alpha_2+1)r} ||\xi|-|\tau||^{2-(\alpha_1+1)r} \sim |\xi|^{(\alpha_0-\alpha_2)r} ||\tau|-|\xi||^{2-\alpha_1 r} \lesssim 1 \, , $$
by the assumption $\alpha_1 \le \frac{2}{r}$ .\\
Subcase: $|\eta|+|\xi-\eta| > 2|\xi|$ . \\
By \cite{FK}, Lemma 4.4 we obtain
\begin{align*}
&\int_{|\eta|+|\xi-\eta| > 2|\xi|} \frac{\delta(\tau-|\eta|+|\xi-\eta|)}{|\eta|^{(\alpha_1+1)r} |\xi-\eta|^{(\alpha_2+1)r}} d\eta 
 \\
& \sim \int_2^{\infty} (|\xi|x+\tau)^{-(\alpha_1+1)r} (|\xi|x-\tau)^{-(\alpha_2+1)r} (|\xi|^2x^2-\tau^2) dx \\
& \sim \int_2^{\infty} (x+ \frac{\tau}{|\xi|})^{-(\alpha_1+1)r+1} (x-\frac{\tau}{|\xi|})^{-(\alpha_2+1)r+1} dx \cdot |\xi|^{-(\alpha_1+\alpha_2+2)r+2} \, . 
\end{align*}
Using $|\tau| \le |\xi|$ and the assumption $\alpha_1+\alpha_2 > \frac{3}{r}-2$ the integral is bounded. Thus we obtain
$$I \lesssim |\xi|^{(\alpha_0+1)r} ||\tau|-|\xi||^r |\xi|^{-(\alpha_1+\alpha_2+2)r+2} \lesssim |\xi|^{(\alpha_0-\alpha_1-\alpha_2)r+2} = 1 \, . $$
The transfer principle completes the proof.
\end{proof}

\begin{lemma}
\label{Lemma1.1}
Let $r=1+$ , $ b > \frac{1}{r} $ , $\alpha_1 < \frac{3}{r}-1$ , $\alpha_2 > 1$ and $\alpha_1+\alpha_2 > \frac{3}{r}$ . Then the following estimate applies
$$ \|uv\|_{X^r_{0,0}} \lesssim \|  u\|_{\dot{X}^r_{\alpha_1,b}}\| \nabla v\|_{X^r_{\alpha_2-1,b}} \, . $$
\end{lemma}
\begin{proof}
By H\"older and Young we obtain
\begin{align*}
&\|uv\|_{X^r_{0,0}} = \|\widehat{uv}\|_{L^{r'}_{\tau \xi}} \\
&\lesssim
  \||\xi|^{-\alpha_1} \langle|\tau|-|\xi|\rangle^{-b}\|_{L^{r_1}_{\tau} L^{p_2}_{\xi}} \|u \|_{\dot{X}^r_{\alpha_1,b}} \\
& \cdot \| \langle \xi \rangle^{-\alpha_2+1}\|_{L^{q_1}_{\xi}} \||\xi|^{-1} \langle|\tau|-|\xi|\rangle^{-b}\|_{L^{r_2}_{\tau} L^{q_2}_{\xi}} \|\nabla v \|_{X^r_{\alpha_2-1,b}} \\
& \lesssim \|  u\|_{\dot{X}^r_{\alpha_1,b}}\| \nabla v\|_{X^r_{\alpha_2-1,b}}         \, .
\end{align*}
Here we need $1+\frac{1}{r'} = \frac{1}{p_2} + \frac{1}{r'} + \frac{1}{q_1}+\frac{1}{q_2} + \frac{1}{r'}$ and $\frac{1}{r_1}+\frac{1}{r_2}= 1+\frac{1}{r'}$ . Choose $\frac{1}{r_1}=\frac{1}{r_2}  = \half+\frac{1}{2r'}$ , and $p_2 = \frac{3}{\alpha_1}\pm$ , $q_2 = 3\pm$ , where the signs are chosen dependent on the regions  $|\xi|\ge 1$ or $|\xi| \le 1$, so that we obtain $\frac{1}{q_1} = \frac{1}{r} -\frac{1}{3} - \frac{\alpha_1}{3} \pm $ . This requires  $\alpha_1 < \frac{3}{r}-1$  . Moreover we need $q_1(\alpha_2-1) > 3$ , which can be fulfilled if $\alpha_1+\alpha_2 > \frac{3}{r}$ and $\alpha_2 > 1$ .
\end{proof}
As a consequence of these results we obtain the following lemma.

\begin{lemma}
\label{Cor.1}
Let $r=1+$ , $b > \frac{1}{r}$ , and $q(u,v) = q_{ij}(u,v)$ or $q(u,v) = q_{0j}(u,v)$ or $q(u,v) = q_0(u,v)$ . The following estimate applies:
$$ \|q(u,v)\|_{X^r_{1,0}} \lesssim \|\nabla u\|_{X^r_{\frac{2}{r}-1,b}} \|v\|_{X^r_{1,b}} \, . $$
\end{lemma} 
\begin{proof} This follows from Lemma \ref{Prop.1.1}, \ref{Prop.1.2}, \ref{Prop.1.3} with $\alpha_0 =1$, $\alpha_1=\frac{2}{r}$ , $\alpha_2=1$ and Lemma \ref{Lemma1.1} .
\end{proof}

We want to apply this result to the nonlinearity $A^{\mu}\partial_{\mu}\phi$ and recall the known null structure of this term, which can be found in \cite{ST} or \cite{P}. We use the Hodge decomposition $A=A^{cf}+A^{df}$ , where $A^{cf}=\Delta^{-1}\nabla(\nabla \cdot A)$ and $A^{df}=-\Delta^{-1}\nabla \times \nabla \times A$ are the curl-free and divergence-free part, respectively. This decomposes the term as follows:
$$ A^{\mu} \partial_{\mu}\phi = (-A_0 \partial_t \phi + A^{cf} \nabla \phi) + A^{df} \nabla \phi =: P_1+P_2 \, . $$
Now 
$$P_2 = A^{df} \nabla \phi = (\nabla w_l \times \nabla \phi)^l \, , $$
 where $w=\Delta^{-1} \nabla \times A$ . The symbol  of $P_2$ is bounded by $$\frac{|\eta \times (\xi - \eta)|}{|\eta|} \lesssim  |\xi-\eta| \max_{i,j}|\widehat{q_{ij}}(\eta,\xi-\eta)|\,.$$ Using the Lorenz gauge $\partial_t A_0 = \nabla \cdot A$ we obtain $A^{cf}=\Delta^{-1}\nabla \partial_t A_0$, thus
$$ P_1= -A_0 \partial_t \phi+ \Delta^{-1}\nabla \partial_t A_0 \cdot \nabla \phi \, . $$
Recalling $A_{0 \pm} = \half(A_0 \pm (iD)^{-1} A_{0t})$
 and $\phi_{\pm} = \half(\phi \pm (i\Lambda_m)^{-1} \phi_t)$
 this reads as follows:
\begin{align*}
iP_1 &= (A_{0+} +A_{0-}) \Lambda_m(\phi_+-\phi_-) + D^{-1} \nabla(A_{0+}-A_{0-}) \cdot \nabla(\phi_++\phi_-) \\
&= \sum_{\pm_1,\pm_2} \pm_2(A_{0\pm_1} \Lambda_m \phi_{\pm_2} + D^{-1} \nabla(\pm_1 A_{0\pm_1}) \cdot \nabla(\pm_2 \phi_{\pm_2})) \\
&
 = \sum_{\pm_1,\pm_2} a_{\pm_1,\pm_2}(A_{0\pm_1},\phi_{\pm_2}) \, .
\end{align*}
Thus the symbol of $P_1$ behaves as follows (cf. \cite{S}):
$$ |a_{\pm_1,\pm_2}(\eta,\xi-\eta)| \lesssim 1 + |\xi-\eta|(1-\frac{(\pm_1 \eta) \cdot (\pm_2 (\xi-\eta))}{|\eta||\xi-\eta|}) = 1 + |\xi-\eta|q_0(\pm_1 \eta,\pm_2(\xi-\eta)). $$

These bounds for the symbol of $A^{\mu}\partial_{\mu}\phi$ are now combined with Lemma \ref{Cor.1} and Lemma \ref{Lemma1.1}, which implies the following result.
\begin{lemma}
\label{Prop.1.4} For $r=1+$
the following estimate applies:
$$ \|A^{\mu} \partial_{\mu} \phi\|_{X^r_{1,0}} \lesssim \|\nabla A\|_{X^r_{\frac{2}{r}-1,b}} \|\phi\|_{X^r_{2,b}}  \, .$$
\end{lemma}
\vspace{0.5em}

Next we want to estimate the term $Im(\phi \overline{\partial \phi})$ .
\begin{lemma}
\label{Prop.1.5}
Let $1<r\le 2$ , $b > \frac{1}{r}$ , $\alpha_0,\alpha_1,\alpha_2 \ge 0$ , $ \alpha_1 < 1+\frac{2}{r}$ , $\alpha_2 < \frac{2}{r}$ , $(\alpha_1 -1)+\alpha_2 > \frac{3}{r}$ and $(\alpha_0-\alpha_1-\alpha_2+1)r+2=0$ . The following estimate applies:
$$\|\phi \nabla \phi\|_{\dot{X}^r_{\alpha_0,0}}
\lesssim \|\phi\|_{\dot{X}^r_{\alpha_2,b}} \|\nabla \phi\|_{\dot{X}^r_{\alpha_1-1,b}} \, . 
 $$
\end{lemma}

\begin{proof}
Arguing similarly as in the lemmas before we have to estimate in the elliptic case
$$I = |\xi|^{\alpha_0 r} \int \frac{\delta(\tau-|\eta|-|\xi-\eta|)}{|\eta|^{\alpha_2 r} |\xi-\eta|^{(\alpha_1-1) r}} d\eta \sim |\xi|^{\alpha_0 r} \tau^A ||\tau|-|\xi||^B \, , $$
by \cite{FK}, Prop. 4.3, where $A=\max(\alpha_2 r,(\alpha_1-1)r,2)-(\alpha_1+\alpha_2-1)r = 2 -(\alpha_1+\alpha_2-1)r$, $B=2-\max(\alpha_2 r,(\alpha_1-1)r,2) = 0$ , provided $\alpha_1-1 < \frac{2}{r}$ , $\alpha_2 < \frac{2}{r}$ .  Thus
$$ I \sim |\xi|^{\alpha_0 r} \tau^{2-(\alpha_1+\alpha_2-1)r} \lesssim \tau^{(\alpha_0-\alpha_1-\alpha_2+1)r+2} = 1$$
provided  $(\alpha_0-\alpha_1-\alpha_2 +1)r+2=0$ .\\
In the hyperbolic subcase $|\eta|+|\xi-\eta| \le 2|\xi|$ we obtain similarly as in the elliptic case by \cite{FK}, Prop. 4.5:
$$I =|\xi|^{\alpha_0 r} \int_{|\eta|+|\xi-\eta| \le 2|\xi|} \frac{\delta(\tau-|\eta|+|\xi-\eta|)}{|\eta|^{\alpha_2 r} |\xi-\eta|^{(\alpha_1-1)r}} d\eta \sim |\xi|^{\alpha_0 r} |\xi|^{2-\alpha_2r-(\alpha_1-1)r} = 1 \, ,$$
whereas the subcase $|\eta| + |\xi-\eta| > 2|\xi|$ is handled as follows:
\begin{align*}
I & = |\xi|^{\alpha_0 r} \int_{|\eta|+|\xi-\eta| > 2|\xi|} \frac{\delta(\tau-|\eta|+|\xi-\eta|)}{|\eta|^{\alpha_2 r} |\xi-\eta|^{(\alpha_1 -1)r}} d\eta \\
&\sim |\xi|^{\alpha_0 r} \int_2^{\infty} (|\xi|x+\tau)^{-\alpha_2 r +1} (|\xi|x-\tau)^{-(\alpha_1-1)r+1} dx \\
& \sim |\xi|^{\alpha_0 r} \int_2^{\infty}
(x+ \frac{\tau}{|\xi|})^{-\alpha_2 r +1}
 (x-\frac{\tau}{|\xi|})^{-(\alpha_1-1)r+1} dx 
\cdot |\xi|^{2-(\alpha_1+\alpha_2-1)r} 
 \\
& \lesssim |\xi|^{2+(\alpha_0-\alpha_1-\alpha_2+1)r} = 1
\end{align*}
for $(\alpha_0-\alpha_1-\alpha_2+1)r+2 =0$ , where the integral is bounded, because $|\tau| \le |\xi|$ and $(\alpha_1 -1)+\alpha_2 > \frac{3}{r}$ by assumption. The transfer principle completes the proof.
\end{proof}

\begin{Cor}
\label{Cor.1.1}
For $r=1+$ , $b > \frac{1}{r}$ and $\epsilon > 0$ the following estimate applies:
$$ \|\phi \nabla \phi\|_{X^r_{1-\epsilon,0}} \lesssim \|\phi\|_{X^r_{2,b}} \|\nabla \phi\|_{X^r_{1,b}} \, . $$
\end{Cor}
\begin{proof}
We use Lemma \ref{Prop.1.5} with $\alpha_0=1-\epsilon$ , $\alpha_2=\frac{2}{r}-\epsilon$ , $\alpha_1=2$ , which gives 
$$ \|\phi \nabla \phi\|_{\dot{X}^r_{1-\epsilon,0}} \lesssim \|\phi\|_{\dot{X}^r_{\frac{2}{r}-\epsilon,b}} \|\nabla \phi\|_{X^r_{1,b}} \lesssim \|\phi\|_{X^r_{2,b}} \|\nabla \phi\|_{X^r_{1,b}}\, , $$
and Lemma \ref{Lemma1.1} with $\alpha_1=1$ , $\alpha_2=2$ .
\end{proof}

Next we want to handle the cubic term $A_{\mu}A^{\mu} \phi$ . We prepare this by the following lemma.
\begin{lemma}
\label{Lemma1.2}
For $r=1+$ , $b > \frac{1}{r}$ and $0 < \epsilon < \frac{3}{2}(1-\frac{1}{r})$ the following estimate applies:
$$\|uv\|_{X^r_{1+\epsilon,b}} \lesssim \|\nabla u\|_{X^r_{1-\epsilon,b}} \| v\|_{X^r_{2,b}} \, . $$
\end{lemma}
\begin{proof} By Lemma \ref{Lemma1.1} we obtain
\begin{equation}
\label{*}
\|uv\|_{X^r_{1+\epsilon,0}} \lesssim \|\nabla u\|_{X^r_{1-\epsilon,b}} \| v\|_{X^r_{2,b}}   \, .
\end{equation}
We reduce namely by the fractional Leibniz rule to the estimates
$$\|uv\|_{X^r_{0,0}} \lesssim \|\nabla u\|_{X^r_{1-\epsilon,b}} \|v\|_{X^r_{1-\epsilon,b}} $$
and
 $$\|uv\|_{X^r_{0,0}} \lesssim \| u\|_{\dot{X}^r_{1-2\epsilon,b}} \|\nabla v\|_{X^r_{1,b}} \, .$$
The first inequality follows by Lemma \ref{Lemma1.1} with $\alpha_1=1-\epsilon$ , $\alpha_2 = 2-\epsilon$ , so that $\alpha_1+\alpha_2 > \frac{3}{r}$ by the assumption $\epsilon < \frac{3}{2}(1-\frac{1}{r}) $ . The second inequality follows similarly with parameters $\alpha_1=1-2\epsilon$, $\alpha_2=2$ .

Next we apply the "hyperbolic Leibniz rule" (cf. \cite{AFS2}, chapter 4.2) :
$$||\tau|-|\xi|| \le ||\rho|-|\eta|| +||\tau-\rho| -|\xi-\eta|| + b_{\pm}(\xi,\eta) \, . $$
where 
$$ b_+(\xi,\eta) = |\eta|+|\xi-\eta|-|\xi| \quad , \quad b_-(\xi,\eta) = |\xi|-||\eta|-|\xi-\eta|| \, , $$
so that
\begin{equation}
\label {**}
\|D_-^b(uv) \|_{X^r_{1+\epsilon,0}} \lesssim \|B_{\pm}^b(uv)\|_{X^r_{1+\epsilon,0}} + \|(D_-^b u)v\|_{X^r_{1+\epsilon,0}} + \|u (D_-^b v)\|_{X^r_{1+\epsilon,0}} \, ,
\end{equation}
where $\widehat{D_-^b w}(\xi,\tau) = ||\xi|-|\tau||^b \widehat{w}(\xi,\tau)$ and $ \widehat{B_{\pm}^b w}(\xi,\tau) = b_{\pm}^b(\xi,\tau) \widehat{w}(\xi,\tau)$ . By Young and H\"older we obtain
\begin{align*}
&\|(D_-^b D^{1+\epsilon} u)v\|_{X^r_{0,0}} \lesssim \| (|\xi|^{1+\epsilon} \langle |\tau|-|\xi| \rangle^b \widehat{u}) \ast \widehat{v} \|_{L^{r'}_{\tau \xi}} \\
& \lesssim \| \langle \xi \rangle^{-(1-2\epsilon)}\|_{L^p_{\xi}} \| |\xi|^{1+\epsilon} \langle \xi \rangle^{1-2\epsilon} \langle|\tau|-|\xi|\rangle^b \widehat{u}\|_{L^{r'}_{\tau\xi}} \\
&\quad \cdot\|\langle|\tau|-|\xi|\rangle^{-b} |\xi|^{-1}\|_{L^r_{\tau} L^q_{\xi}} \| \langle \xi \rangle	^{-1}\|_{L^s_{\xi}} \| \langle|\tau|-|\xi|\rangle^b |\xi| \langle \xi \rangle \widehat{v}\|_{L^{r'}_{\tau \xi}} \\
& \lesssim \| \nabla u\|_{X^r_{1-\epsilon,b}} \|\nabla v\|_{X^r_{1,b}} \, .
\end{align*}
Here we need $1+\frac{1}{r'} = \frac{1}{p}+\frac{1}{r'} + \frac{1}{q}+\frac{1}{s}+\frac{1}{r'}\, \Leftrightarrow \, \frac{1}{p}+\frac{1}{q}+\frac{1}{s}=1-\frac{1}{r '}$. We choose $p= \frac{3}{1-2\epsilon}+$, $q=3\pm$ , so that $\frac{1}{s} = 1-\frac{1}{r'} - \frac{1-2\epsilon}{3} - \frac{1}{3} \pm = \frac{1}{3}-\frac{1}{r'}+\frac{2}{3}\epsilon \pm = \frac{1}{r}-\frac{2}{3}+\frac{2}{3}\epsilon \pm < \frac{1}{3}$ provided $\epsilon < \frac{3}{2}(1-\frac{1}{r})$, so that $s > 3$ , where + is chosen for $|\xi|\ge 1$ and -- for $|\xi| \le 1$ , so that all the integrals exist. Similarly we can also handle the term $\|(D_-^b  u)(D^{1+\epsilon}v)\|_{X^r_{0,0}}$ , so that by the fractional Leibniz rule we obtain the desired estimate for  $\|(D_-^b u)v\|_{\dot{X}^r_{1+\epsilon,0}} $ . Because the term $\|(D_-^b u)v\|_{X^r_{0,0}} $  can easily be estimated similarly, we obtain
$$ \|(D_-^b u) v\|_{X^r_{1+\epsilon,0}} \lesssim \|\nabla u\|_{X^r_{1-\epsilon,b}} \|\nabla v\|_{X^r_{1,b}} \, .$$ 
 By similar arguments we also obtain this bound for the term $\|u(D_-^b v)\|_{X^r_{1+\epsilon,0}}$ .

It remains to prove
\begin{equation}
\label{***}
 \|B_{\pm}^b(uv)\|_{X^r_{1+\epsilon,0}} \lesssim \|\nabla u\|_{X^r_{1-\epsilon,b}} \|\nabla v\|_{X^r_{1,b}} \, . 
\end{equation}
For the + -sign we are in the elliptic case and prove for suitable $l \le 2-\epsilon$ :
$$ I = |\xi|^{(1+\epsilon)r} ||\tau|-|\xi||^{br} \int \delta(\tau-|\eta|-|\xi-\eta|) |\eta|^{-lr} |\xi-\eta|^{-2r} d\eta \lesssim 1 \, . $$
By \cite{FK}, Prop. 4.3 we obtain by $\tau = |\eta|+|\xi-\eta| \ge |\xi|$ :
\begin{align*}
I &\sim |\xi|^{(1+\epsilon)r} ||\tau|-|\xi||^{br} \tau^A ||\tau|-|\xi||^B = |\xi|^{(1+\epsilon)r} ||\tau|-|\xi||^{br} \tau^{-lr} ||\tau|-|\xi||^{2-2r} \\
& \lesssim \tau^{br+2-(l+1)r+\epsilon r} =1 \, ,
\end{align*} 
because $A=\max(lr,2r,2) -(l+2)r = -lr$ , $B=2-2r$ for $l \le 2$ , and if we assume $l=\frac{2}{r}-1+b+\epsilon \le 2-\epsilon $ , which is fulfilled for $\epsilon < \frac{3}{2}(1-\frac{1}{r})$ .  As in Lemma \ref{Prop.1.1} this implies (\ref{***}) by the transfer principle.

For the -- -sign we are in the hyperbolic case. In the subcase $|\eta|+|\xi-\eta| \le 2|\xi|$  we have to prove
$$ I = I =|\xi|^{(1+\epsilon)r} ||\tau|-|\xi||^{br} \int_{|\eta|+|\xi-\eta| \le 2|\xi|} \delta(\tau-|\eta|+|\xi-\eta|) |\eta|^{-lr} |\xi-\eta|^{-2r} d\eta \lesssim 1 \, . $$
Now by \cite{FK}, Prop. 4.5 we obtain similarly as in the elliptic case in the region $0 \le \tau \le |\xi|$ :
\begin{align*}
I &\sim |\xi|^{(1+\epsilon)r} ||\tau|-|\xi||^{br} |\xi|^A ||\tau|-|\xi||^B = ||\tau|-|\xi||^{br+2-2r} |\xi|^{(1+\epsilon)r -lr} \\
&\lesssim |\xi|^{br+2-(l+1)r+\epsilon r} =1 \, ,
\end{align*} 
because $A=\max(2r,2) -(l+2)r = -lr$ , $B=2-2r$ , if we assume $l=\frac{2}{r}-1+b+\epsilon $. In the region $-|\xi| \le \tau < 0$ we obtain similarly with $A=\max(lr,2)-(l+2)r =2-(l+2)r$ and $B=0$ the same estimate.
In the subcase $|\eta|+|\xi-\eta| > 2|\xi|$ we use \cite{FK}, Lemma 4.4, which implies:
\begin{align*}
I &\sim |\xi|^{(1+\epsilon)r} ||\tau|-|\xi||^{br} \int_2^{\infty} (x+\frac{\tau}{|\xi|})^{-2r+1} (x+\frac{\tau}{|\xi|})^{-lr+1} dx \cdot |\xi|^{-(l+2)r+2} \\
&\lesssim |\xi|^{(1+\epsilon)r +br-(l+2)r+2} = 1 \, . 
\end{align*}
The integral converges, provided $(l+2)r -2 > 1 \, \Leftrightarrow \, l >\frac{3}{r}-2$ , which is the case. Thus we obtain (\ref{***})  by the transfer principle.
\end{proof}

Now we are prepared to handle the cubic terms.
\begin{lemma}
\label{Prop.1.6}
Let $r=1+$ and $b > \frac{1}{r}$ . The following estimates apply:
\begin{align*}
\|A_{\mu}A^{\mu} \phi\|_{X^r_{1,0}} &\lesssim \|\nabla A\|^2_{X^r_{1-\epsilon,b}} \|\phi\|_{X^r_{2,b}} \, , \\
\| A(|\phi|^2)\|_{X^r_{1,0}} & \lesssim \|\nabla A\|_{X^r_{1-\epsilon,b}} \|\phi\|^2_{X^r_{2,b}}
\end{align*}
for $0 < \epsilon < \frac{3}{2}(1-\frac{1}{r})$ .
\end{lemma}
\begin{proof} 
We obtain by  Lemma \ref{Prop.1.5} with $\alpha_0 =1$ , $\alpha_1=2+\epsilon$ , $\alpha_2 = \frac{2}{r}-\epsilon$ :
$$\|A_{\mu} A^{\mu} \phi\|_{\dot{X}^r_{1,0}} \lesssim \|A\|_{\dot{X}^r_{\frac{2}{r}-\epsilon,b}} \|A \phi\|_{\dot{X}^r_{1+\epsilon,b}} $$
and by Lemma \ref{Lemma1.1} :
$$\|A_{\mu} A^{\mu} \phi\|_{{X}^r_{0,0}} \lesssim \|\nabla A\|_{{X}^r_{1-\epsilon,b}} \|\nabla(A \phi)\|_{\dot{X}^r_{\epsilon,b}} \, . $$
This implies
\begin{align*}
\|A_{\mu} A^{\mu} \phi\|_{{X}^r_{1,0}} &\lesssim \|\nabla A\|_{{X}^r_{1-\epsilon,b}} \|\nabla(A \phi)\|_{\dot{X}^r_{\epsilon,b}} \\
& \lesssim \|\nabla A\|_{{X}^r_{1-\epsilon,b}}^2 \| \phi\|_{X^r_{2,b}} \, ,
\end{align*}
where we applied Lemma \ref{Lemma1.2} in the last step.
The second estimate is proven in the same way.
\end{proof}

We now consider the nonlinearities in the equations (\ref{27}) and (\ref{28}) for $F_{\mu \nu}$ .
\begin{lemma}
\label{Prop. 1.7}
Let $r=1+$ and $ b > \frac{1}{r}$ . The following estimates apply:
\begin{align*}
\| \nabla \phi \times \overline{\nabla \phi}\|_{X^r_{0,0}} &\lesssim \|\nabla \phi\|_{X^r_{1,b}}^2 \, , \\
\|\partial_t \phi \overline{\nabla \phi} - \nabla \phi \overline{\partial_t \phi}\|_{X^r_{0,0}} & \lesssim \|\partial_t \phi\|_{X^r_{1,b}} \|\nabla \phi\|_{X^r_{1,b}} \, .
\end{align*}
\end{lemma}
\begin{proof}
For the first estimate we use Lemma \ref{Prop.1.1} with $\alpha_0 =0$ , $\alpha_1=\alpha_2= \frac{1}{r} < 1$ and $\alpha_1+\alpha_2 > \frac{3}{r}-1$ . The second estimate results from Lemma \ref{Prop.1.2}.
\end{proof}

The cubic terms are handled by the following lemma.
\begin{lemma}
\label{Prop.1.8}
For $ r = 1+$ and $b > \frac{1}{r}$ the following estimates are true:
\begin{align*}
\|\nabla (A|\phi|^2)\|_{X^r_{0,0}} & \lesssim \|\nabla A\|_{X^r_{1-\epsilon,b}} \|\phi\|_{X^r_{2,b}}^2 \, , \\
\|\partial_t (A|\phi|^2)\|_{X^r_{0,0}} & \lesssim \|\nabla A\|_{X^r_{1-\epsilon,b}}  \|\phi\|_{X^r_{2,b}}^2
\end{align*}
for $0<\epsilon < \frac{3}{2}(1-\frac{1}{r})$ . 
\end{lemma}
\begin{proof}
The first  estimate is proven exactly like Lemma \ref{Prop.1.6}. For the last estimate we modify the proof of Lemma \ref{Prop.1.5} for the case $\alpha_0=1$ in order to prove:
$$\|\partial_t(\phi \nabla \phi)\|_{X^r_{0,0}} \lesssim \|\phi\|_{\dot{X}^r_{\frac{2}{r}-\epsilon,b}} \|\nabla \phi\|_{\dot{X}^r_{1+\epsilon,b}} \, . $$
We simply replace $|\xi|^{\alpha_0 r} = |\xi|^r$ by $\tau^r$ and obtain the same estimate for $I$ in the elliptic case, whereas in the hyperbolic case we use $|\tau| \le |\xi|$ , which leads to the previous bounds for $I$. Using this new estimate we may handle the term $\|\partial_t(A|\phi|^2)\|_{X^r_{0,0}}$ exactly like $\|A |\phi|^2\|_{\dot{X}^r_{1,0}} $ in Lemma \ref{Prop.1.6}.
\end{proof}

Now we interpolate the bi- and trilinear estimates in $X^r_{s,b}$-spaces for $r=1+$ just proven and for $r=2$, where for the latter we rely on the results in \cite{P}.

Let $\delta > 0$ be given and $s=\frac{5}{2r}-\half+\delta$ , $ l=\frac{3}{r}-1+\delta$ . Then for $r>1$ sufficiently close to 1 we have $\delta \ge \frac{5}{2}-\frac{5}{2r}$ , so that $\delta = \frac{5}{2}-\frac{5}{2r} + \omega $ with
$\omega \ge 0$ . For $\omega =0$ we have $s=2$ , $l=\frac{3}{2}+\frac{1}{2r}$ .

\begin{lemma}
\label{Prop.1.9}
Let $1 < r \le 2$ , $ b = \frac{1}{r}+$ , $\delta > 0$ , $s=\frac{5}{2r} - \half+\delta$ , $l=\frac{3}{r}-1+\delta$ .The following estimate applies:
$$\|A^{\mu} \partial_{\mu} \phi\|_{X^r_{s-1,b-1+}} \lesssim \| \nabla A\|_{X^r_{l-1,1-}} \|\phi\|_{X^r_{s,b}} \, . $$
\end{lemma}
\begin{proof}
The result for $r=2$ is given in \cite{P}, Chapter 5, Claim 1:
$$\|A^{\mu} \partial_{\mu} \phi\|_{X^2_{-\frac{1}{4}+\epsilon,-\half++}} \lesssim \| \nabla A\|_{X^2_{-\half+\epsilon,1-}} \|\phi\|_{X^2_{\frac{3}{4}+\epsilon,\half+}}  $$
for arbitrary $ \epsilon > 0$ .
For $r=1+$ , $b > \frac{1}{r}$ and $s=2$ , $l=\frac{3}{2}+\frac{1}{2r}$ we use Lemma \ref{Prop.1.4}:
$$\|A^{\mu} \partial_{\mu} \phi\|_{X^r_{s-1,0}} \lesssim \| \nabla A\|_{X^r_{\frac{2}{r}-1,b}} \|\phi\|_{X^r_{s,b}} \lesssim \| \nabla A\|_{X^r_{l-1,b}} \|\phi\|_{X^r_{s,b}}         \, , $$
because $\frac{2}{r}-1 \le \half+\frac{1}{2r} =l-1$ . By the fractional Leibniz rule this inequality remains true for $ \omega > 0$ , thus for the given $\delta$ and $r > 1$ close enough to 1.
By bilinear complex interpolation with interpolation parameter $\theta = 2-\frac{2}{r}$ we obtain in the whole range $1 < r \le 2$ and $b > \frac{1}{r}$ the claimed estimate. 
\end{proof}

\begin{lemma}
\label{Prop.1.10}
Let $r,b,s,l$ be given as in Lemma \ref{Prop.1.9}. The following estimate applies:
$$ \|\phi \nabla \phi\|_{X^r_{l-1,0}} +  \|\phi \Lambda_m \phi\|_{X^r_{l-1,0}}\lesssim \|\phi\|^2_{X^r_{s,b}} \, . $$
\end{lemma}
\begin{proof}
The estimate for $r=2$ is given by \cite{P}, Chapter 5, Claim 3:
$$ \|D^{-1}(\phi \nabla \phi)\|_{X^2_{\half+\epsilon,0}} \lesssim \|\phi\|^2_{X^2_{\frac{3}{4}+\epsilon,\half+}} \, .$$ 
For $r=1+$ , $b > \frac{1}{r}$ and $s=2$ , $l=\frac{3}{2}+\frac{1}{2r}$ we obtain by Cor. \ref{Cor.1.1}: 
$$\|\phi \nabla \phi\|_{X^r_{l-1,0}} \lesssim \|\phi \nabla \phi\|_{X^r_{1-\epsilon,0}} \lesssim \|\phi\|_{X^r_{s,b}}^2 \, , $$
where we remark that $\nabla \phi$ may be replaced by $\Lambda_m \phi$ . As before interpolation gives the claimed estimate.
\end{proof}

\begin{lemma}
\label{Prop.1.11}
For $r,b,s,l$ as in Lemma \ref{Prop.1.9} the following estimates apply:
\begin{align*}
\|A_{\mu}A^{\mu} \phi\|_{X^r_{s-1,b-1+}} & \lesssim \|\nabla A\|_{X^r_{l-1,b}}^2 \|\phi\|_{X^r_{s,b}} \, , \\
\|A |\phi|^2\|_{X^r_{l-1,0}} & \lesssim \|\nabla A\|_{X^r_{l-1,b}} \|\phi\|_{X^r_{s,b}}^2 \, .
\end{align*}
\end{lemma} 
\begin{proof}
We interpolate between \cite{P}, Chapter 5, Claim 4 and Claim 5 and
Lemma \ref{Prop.1.6}, which implies for $ r=1+$ , $s=2$ , $l=\frac{3}{2}+\frac{1}{2r}$ :
$$\|A_{\mu}A^{\mu} \phi\|_{X^r_{s-1,b-1+}}  \lesssim \|\nabla A\|_{X^r_{1-\epsilon,b}}^2 \|\phi\|_{X^r_{s,b}} \lesssim \|\nabla A\|_{X^r_{l-1,b}}^2 \|\phi\|_{X^r_{s,b}} \, ,$$
if $1-\epsilon \le l-1=\frac{1}{2}+\frac{1}{2r} \, \Leftrightarrow \, \epsilon \ge \half - \frac{1}{2r}$ , which is admissible by Lemma \ref{Prop.1.6}. Moreover Lemma \ref{Prop.1.8} implies similarly
$$\|A |\phi|^2\|_{X^r_{l-1,0}}\lesssim \|A |\phi|^2\|_{X^r_{1,0}} \lesssim \|\nabla A\|_{X^r_{1-\epsilon,b}} \|\phi\|_{X^r_{s,b}}^2 \lesssim \|\nabla A\|_{X^r_{l-1,b}} \|\phi\|_{X^r_{s,b}}^2 \, . $$
As before interpolation completes the proof.
\end{proof}

\begin{lemma}
\label{Prop.1.12}
For  $r,b,s,l$ as in Lemma \ref{Prop.1.9} the following estimates apply:
\begin{align*}
\| \partial_t \phi \overline{\partial_k \phi} - \partial_k \phi \overline{\partial_t \phi}\|_{X^r_{s-2,b-1+}} & \lesssim \|\phi\|_{X^r_{s,b}} \|\partial_t \phi\|_{X^r_{s-1,b}} \, , \\
\| \nabla \phi \times \overline{\nabla \phi} \|_{X^r_{s-2,b-1+}} & \lesssim \|\phi\|_{X^r_{s,b}}^2  \, , \\
\|\partial_t(A_k |\phi|^2)\|_{X^r_{s-2,b-1+}} & \lesssim \|\nabla A_k\|_{X^r_{l-1,b}} \|\phi\|^2_{X^r_{s,b}} \\
\|\nabla(A |\phi|^2)|\|_{X^r_{s-2,b-1+}} & \lesssim \|\nabla A\|_{X^r_{l-1,b}} \|\phi\|^2_{X^r_{s,b}} \, .
\end{align*}
\end{lemma}
\begin{proof}
Concerning the first estimate we combine the following estimate by \cite{P}, Chapter 6 ( more precisely we use Claim 1 and inequality (39),(40) in \cite{P}):
$$ \|D^{-1}( \partial_t \phi \overline{\partial_k \phi} - \partial_k \phi \overline{\partial_t \phi})\|_{X^2_{-\frac{1}{4}+\epsilon,-\half+}} \lesssim \|\phi\|_{X^2_{\frac{3}{4}+\epsilon,\half+}} \|\partial_t \phi\|_{X^2_{-\frac{1}{4}+\epsilon,\half+}} $$
and Lemma \ref{Prop.1.2} for $r=1+$ , $\alpha_0=0$ , $\alpha_1,\alpha_2 = \frac{1}{r}$ :
$$\|  \partial_t \phi \overline{\partial_k \phi} - \partial_k \phi \overline{\partial_t \phi} \|_{X^r_{0,0}}  \lesssim \|\nabla \phi\|_{X^r_{1,b}} \|\partial_t \phi\|_{X^r_{1,b}} \, .$$
Similarly the second estimate is proven by \cite{P}, Chapter 6 and Lemma \ref{Prop.1.1}. Concerning the remaining estimates we interpolate the case $r=2$ by \cite{P}, Chapter 6, Claims 3 and 5, and Lemma \ref{Prop.1.8}.
\end{proof}

\section{Proof of the Theorems}
\begin{proof}[Proof of Theorem \ref{Theorem0.1}]
By Theorem \ref{Theorem0.0} the claimed result follows by the contraction mapping principle, if the following estimates apply for the system (\ref{16}),(\ref{17}):
\begin{equation}
\label{51}
\| \Lambda^{-1} M(\phi_+,\phi_-,A_+,A_-)\|_{X^r_{s,b-1+,\pm}} \lesssim R^2 + R^3
\end{equation}
and
\begin{equation}
\label{52}
\|\mathcal{N}(\phi_+,\phi_-,A_+,A_-)\|_{X^r_{l-1,-\epsilon_0+,\pm}} \lesssim R^2 + R^3 \, ,
\end{equation}
where
$$ R= \sum_{\pm}\left(\|\phi_{\pm}\|_{X^r_{s,b,\pm}} + \|\nabla A_{\pm}\|_{X^r_{l-1,1-\epsilon_0,\pm}}\right) \, . $$
The estimate (\ref{51}) is a consequence of Lemma \ref{Prop.1.9} and Lemma \ref{Prop.1.11}, whereas (\ref{52}) follows from Lemma \ref{Prop.1.10} and Lemma \ref{Prop.1.11}.
\end{proof}

{\bf Remark:} We make use of the following well-known estimate for the linear wave equation $ \square F = G $ . Considering the equivalent first order equation
$ (i\partial_t \pm D) F_{\pm} = -(\pm 2 D)^{-1} G $ , where $F_{\pm}=\half(F\pm(iD)^{-1} F_t)$ , so that $F=F_++F_-$ and $\partial_t F = iD(F_+-F_-)$, we may use \cite{G}, Chapter 2.2 to conclude
$$ \|\nabla F_{\pm}\|_{X^r_{s-2,b,\pm}[0,T]} \lesssim \|(\nabla F)(0)\|_{\widehat{H}^{s-2,r}} + \|(\partial_t F)(0)\|_{\widehat{H}^{s-2,r}} + T^{0+} \|G\|_{X^r_{s-2,b-1+,\pm}[0,T]} \, , $$
which immediately implies
\begin{align}
\label{LE}
&\|\nabla F\|_{X^r_{s-2,b}[0,T]} + \|\partial_t F\|_{X^r_{s-2,b}[0,T]}  \\
\nonumber
&     \lesssim \|(\nabla F)(0)\|_{\widehat{H}^{s-2,r}} + \|(\partial_t F)(0)\|_{\widehat{H}^{s-2,r}} + T^{0+} \|G\|_{X^r_{s-2,b-1+}[0,T]} \, . 
\end{align}

\begin{proof}[Proof of Theorem \ref{Theorem0.2}]
By the linear estimate (\ref{LE}) we have to prove the following properties:
\begin{align}
\label{61}
(\nabla F_{\mu \nu})(0) & \in \widehat{H}^{s-2,r} \\
\label{62}
(\partial_t F_{\mu\nu})(0) &\in \widehat{H}^{s-2,r} \\
\label{63}
\square F_{\mu\nu} & \in X^r_{s-2,b-1+}[0,T] \, .
\end{align}
(\ref{63}) is a consequence of Lemma \ref{Prop.1.12} and Theorem \ref{Theorem0.1}, (\ref{61}) is our assumption (\ref{8}), so that it remains to prove (\ref{62}).
By (\ref{1}) we obtain
\begin{equation}
\label{101}
\partial_t F_{0k} = -\partial_t F_{k0} = -\partial^l F_{kl} + j_k \, . 
\end{equation}
Now we have $ j_{k_{|t=0}} = Im(\phi_0 \overline{\partial_k \phi_0}) + |\phi_0|^2 a_{0k} $ . In the case $r=1+$ we may estimate by Young and H\"older :
\begin{align*}
&\|\phi_0 \overline{\partial_k \phi_0}\|_{\widehat{H}^{0,r}} = \|\widehat{\phi_0 \overline{\partial_k \phi_0}}\|_{L^{r'}}\\
& \lesssim \|\langle \xi \rangle^{-1}\|_{L^p} \| |\xi|^{-1}\|_{L^{3\pm}} \||\xi|\langle \xi \rangle \widehat{\phi_0}\|_{L^{r'}} \|\langle \xi \rangle^{-1}\|_{L^{3+}}  \| \langle \xi \rangle \widehat{\partial_k \phi_0}\|_{L^{r'}} \\
& \lesssim \| \nabla \phi_0\|_{\widehat{H}^{1,r}} \|\partial_k \phi_0\|_{\widehat{H}^{1,r}} \, .
\end{align*}
Here we need  $1+\frac{1}{r'}= \frac{2}{3}+\frac{1}{p}+\frac{2}{r'} - $ , so that $\frac{1}{p} = \frac{1}{3}-\frac{1}{r'} + < \frac{1}{3}$ , so that the integrals converge. $3+$ and $3-$ belongs to the region $|\xi|\ge 1$ and $|\xi| \le 1$, respectively. In the case $r=2$ we obtain by Sobolev's multiplication law:
$$\|\phi_0 \overline{\partial_k \phi_0}\|_{H^{s-2}} \lesssim \|\phi_0\|_{H^s} \|\partial_k \phi_0\|_{H^{s-1}} $$
for $s>\half$ , especially for $s=\frac{3}{4}+$ , so that by interpolation we obtain
$$\|\phi_0 \overline{\partial_k \phi_0}\|_{\widehat{H}^{s-2,r}} \lesssim \|  \phi_0\|_{\widehat{H}^{s,r}} \|\partial_k \phi_0\|_{\widehat{H}^{s-1,r}} < \infty \, . $$
Moreover we obtain in the case $r=1+$ similarly as before:
$$ \| |\phi_0|^2 a_{0k}\|_{\widehat{H}^{0,r}} \lesssim \| |\phi_0|^2\|_{\widehat{H}^{1,r}} \|\nabla a_{0k}\|_{\widehat{H}^{1,r}} \lesssim \|\phi_0\|_{\widehat{H}	^{2,r}}^2 \|\nabla a_{0k}\|_{\widehat{H}^{1,r}} \, . $$
In the case $r=2$ , $s=\frac{3}{4}+\delta$ , $l= \half+\delta$ we use Sobolev's multiplication law (cf. \cite{T}, Prop. 3.15). If the frequency of $\widehat{a_{0k}}$ is small we obtain
$$ \| |\phi_0|^2 a_{0k}\|_{H^{s-2}} \le \| |\phi_0|^2 a_{0k}\|_{\dot{H}^{-\half}} \lesssim \| |\phi_0|^2\|_{L^2} \|a_{0k}\|_{\dot{H}^1} \lesssim \|\phi_0\|_{H^s}^2 \|\nabla a_{0k}\|_{H^{l-1}} \, , $$ 
and in the case of high frequencies of $\widehat{a_{0k}}$ :
$$ \| |\phi_0|^2 a_{0k}\|_{H^{s-2}} \lesssim \| |\phi_0|^2\|_{L^2} \|a_{0k}\|_{H^l} \lesssim \|\phi_0\|_{H^s}^2 \|\nabla a_{0k}\|_{H^{l-1}}  $$ 
first for $s \le \frac{3}{2}$ , but then also for larger $s$ by the Leibniz rule.
By interpolation we obtain 
$$ \| |\phi_0|^2 a_{0k}\|_{\widehat{H}^{s-2,r}} \lesssim \|\phi_0\|^2_{\widehat{H}^{s,r}} \|\nabla a_{0k}\|_{\widehat{H}^{l-1,r}} < \infty \, . $$
This implies $j_{k_{|t=0}} \in \widehat{H}^{s-2,r}$ , thus by (\ref{101}) $\partial_t F_{{0k}_{|t=0}} \in  \widehat{H}^{s-2,r}$ , because by (\ref{8}) : $F_{kl} \in \widehat{H}^{s-1,r}$ . Finally we obtain :
$$\partial_t F_{jk} = \partial_j(\partial_k A_0 + F_{0k}) - \partial_k(\partial_j A_0 + F_{0j}) = \partial_j F_{0k} - \partial_k F_{0j} \, ,$$
so that $\partial_t F_{{jk}_{|t=0}} \in \widehat{H}^{s-2,r}$ by (\ref{8}) , thus (\ref{62}) is proven.
\end{proof}

\begin{proof}[Proof of Theorem \ref{Theorem0.3}]
Starting with the solution $(\phi_{\pm},A_{\pm})$ of Theorem \ref{Theorem0.1} and defining $\phi=\phi_+ + \phi_-$ , $A=A_++A_-$ , it is possible to show that $(\phi,A)$ fulfills (\ref{1}),(\ref{2}) and the Lorenz condition. Moreover $F_{\mu \nu} = \partial_{\mu}A_{\nu} - \partial_{\nu} A_{\mu}$ fulfills (\ref{27}) and (\ref{28}), so that the claimed regularity follows by Theorem \ref{Theorem0.1} and Theorem \ref{Theorem0.2}.  Because these facts were proven in \cite{P}, Section 6 (see also  \cite{ST}, Section 5) we omit the calculation.
\end{proof}
{\bf Acknowledgment:} I thank the referees for their proposals which helped to improve and modify the paper.

\end{document}